\documentclass[reqno]{amsart}
\usepackage{amssymb,amsmath,amsthm}
\usepackage[colorlinks=true]{hyperref}
\hypersetup{urlcolor=blue, citecolor=blue}
\usepackage[dvips]{graphicx}
\usepackage{color}
\usepackage{comment}

\usepackage{soul,color}
\usepackage{marginnote}
\usepackage{tabularx}

\theoremstyle{plain}

\newtheorem{mainthm}{Theorem}

\newtheorem*{conj*}{Conjecture}
\newtheorem*{cor*}{Corollary}
\newtheorem{theorem}{Theorem}[section]

\newtheorem{corollary}[theorem]{Corollary}
\newtheorem{lemma}[theorem]{Lemma}

\theoremstyle{definition}
\newtheorem*{def*}{Definition}
\newtheorem{remark}[theorem]{Remark}

\newtheorem{definition}[theorem]{Definition}

\newcommand{\floor}[1]{\left\lfloor #1 \right\rfloor}

\newcommand{\SA}{{\mathcal A}}

\newcommand{\SC}{{\mathcal C}}

\newcommand{\SE}{{\mathcal E}}
\newcommand{\SF}{{\mathcal F}}

\newcommand{\SK}{{\mathcal K}}

\newcommand{\SM}{{\mathcal M}}

\newcommand{\SR}{{\mathcal R}}
\renewcommand{\SS}{{\mathcal S}}

\newcommand{\SU}{{\mathcal U}}

\newcommand{\ExpF}{\SE^\phi}
\newcommand{\CExpF}{\tilde{\SE}^\phi}

\newcommand{\al} {\alpha}       
        
\newcommand{\ga} {\gamma}    \newcommand{\Ga}{\Gamma}
\newcommand{\de} {\delta}       \newcommand{\De}{\Delta}

\newcommand{\vep}{\varepsilon}
\renewcommand{\epsilon}{\varepsilon}

\newcommand{\te} {\theta}

\newcommand{\si} {\sigma}

\newcommand{\supp}{\operatorname{supp}}

\newcommand{\Z}{\mathbb{Z}}
\newcommand{\N}{\mathbb{N}}
\newcommand{\R}{\mathbb{R}}
\newcommand{\Q}{\mathbb{Q}}
\newcommand{\eps}{\varepsilon}

\newcommand{\ve}{\varepsilon}

\newcommand{\unif}{\xrightarrow{\operatorname{unif.}}}
\makeatletter
\newcommand{\tpitchfork}{
  \vbox{
    \baselineskip\z@skip
    \lineskip-.52ex
    \lineskiplimit\maxdimen
    \m@th
    \ialign{##\crcr\hidewidth\smash{$-$}\hidewidth\crcr$\pitchfork$\crcr}
  }
}
\makeatother

\title{On the Invariance of Expansive Measures for Flows}

\author{Eduardo Pedrosa}
\address{Instituto de Matem\'atica, Universidade Federal do Rio de Janeiro, Rio de Janeiro, Brazil.}
\email{eduardocp@im.ufrj.br}

\author{Elias Rego}
\address{Faculty of Applied Mathematics, AGH University of Science and Technology, Krakow, Poland.}
\email{rego@agh.edu.pl}

\author{Alexandre Trilles}
\address{Faculty of Mathematics and Computer Science, Jagiellonian University in Krak\'ow, Krakow, Poland.}
\email{alexandre.trilles@im.uj.edu.pl}



\begin{document}

\begin{abstract}
We study expansive measures for continuous flows without fixed points on compact metric spaces. We provide a new characterization of expansive measures through dynamical balls that, in contrast to the dynamical balls considered in [\emph{J. Differ. Equ.}, 256 (2014):2246--2260], are actually Borel sets. This makes the theory more amenable to measure-theoretic analysis. 
We prove that every ergodic invariant measure with positive entropy is positively expansive, extending the results of [\emph{Ergod. Th. \& Dynam. Sys.} \textbf{4}(3) (2014):765--776] to the setting of flows. This implies that flows with positive topological entropy admit expansive invariant measures. Furthermore, we show that the stable classes of such measures have zero measure. Lastly, we prove that the set of expansive measures for a flow is a $G_{\delta\sigma}$-subset of the space of all probability measures and that every expansive measure (invariant or not) can be approximated by expansive measures supported on invariant sets.
\end{abstract}

\maketitle

\section{Introduction}

The concept of expansivity, introduced by Utz \cite{Utz1950}, has become a fundamental notion in the study of dynamical systems. A homeomorphism is called expansive if there exists a constant $\vep > 0$ such that no pair of distinct points can remain within $\vep$ distance of each other along their orbits. This condition formalizes the idea that trajectories of distinct points eventually diverge, even if the points are arbitrarily close. 

Expansivity is deeply connected with hyperbolic dynamics. For instance, expansive systems naturally arise in the context of uniformly hyperbolic systems, such as Anosov diffeomorphisms \cite{KatokHasselblatt}. In these settings, the presence of invariant stable and unstable manifolds guarantees that distinct points diverge exponentially fast, ensuring expansivity. Moreover, expansive systems often exhibit rich dynamical behavior, including positive topological entropy, sensitivity to initial conditions and the existence of measures of maximal entropy. 

Generalizations of expansivity, including the notions of $n$-expansive and countably-expansive homeomorphisms, have also been proposed to capture more general behavior beyond uniform hyperbolicity \cite{ACCV,CC, MS,LMS}. These concepts relax the rigid conditions of classical expansivity, allowing for broader classes of systems to be studied. 

In addition to topological considerations, measure-theoretic analogs of expansivity have gained prominence. Notably, the notion of \emph{expansive measure} was introduced for discrete systems \cite{AM}, inspired by related concepts such as pairwise sensitivity \cite{CadreJacob2003} and $\mu$-sensitivity \cite{Denker2006}. A Borel probability measure $\mu$ is said to be expansive if there is a positive constant $\vep>0$ such that $\mu$ assigns zero measure to any Bowen's ball of diameter $\vep$. For instance, in \cite{AM} those measures are used to study the size of stable classes of diffeomorphisms. Recently, in \cite{OR}, expansive measures were used as a tool for obtaining horseshoe-like sub-dynamics.

In general, ergodic theory focuses on measures that are invariant for the system. For example, invariant measures such as the Bowen-Margulis measure are well understood and play a central role in ergodic theory in hyperbolic settings.
Aside from the usual assumption in ergodic theory, expansive measures are not required to be invariant. A natural question arises: can we find expansive measures that are also invariant under the dynamics? The existence of expansive invariant measures is far from guaranteed in general, and constructing such measures is significantly challenging. In fact, it has been shown that not every expansive measure is invariant, and the invariant ones may be scarce or hard to characterize. Indeed, in \cite{LMS} it is discussed an example of a homeomorphism which is expansive, but does not admit any expansive measure. 

In \cite{AM}, it is shown that positive entropy is a sufficient condition for an ergodic invariant measure to be expansive, implying the existence of expansive invariant measures for every homeomorphism with positive topological entropy. Recently, Lee, Morales, and Shin \cite{LMS} advanced the understanding of this problem from a different perspective by proving that every expansive measure for a homeomorphism on a compact metric space can be approximated in the weak*-topology by expansive measures whose supports are invariant sets for the homeomorphism. These results highlight the connections between expansive measures and invariant structures for discrete-time dynamical systems. 

The concept of expansivity was extended from homeomorphisms to flows by Bowen and Walters \cite{BW}, who showed that many properties of expansive homeomorphisms are preserved by expansive flows. This concept has being widely used to address subtleties specific to continuous-time dynamics, particularly to study the dynamics of hyperbolic-like flows and geodesic flows. 

It is important to note that, although some results in the discrete-time setting can be easily extended to flows, such extension is often challenging. For instance, in the case of expansivity, the difficulty comes from the presence of reparameterizations in the definition, which represents an obstruction to adapt several techniques from expansive homeomorphisms to flows.

Inspired by \cite{MS} and \cite{Ruggiero1996}, Carrasco-Olivera and Morales \cite{CarrascoMorales2014} extended the notion of expansive measures to continuous flows, thereby bridging the discrete and continuous settings. For a continuous flow $\phi\colon\mathbb{R} \times X \to X$ on a metric space $(X,d)$, they defined the generalized dynamical ball centered in $x \in X$ with radius $\vep>0$ as
\[
\tilde{\Ga}_\vep(x) = \bigcup_{h \in \SC^0} \bigcap_{t \in \mathbb{R}} \phi_{-h(t)} \left( B[\phi_t(x), \vep] \right),
\]
where $B[\cdot,\vep]$ denotes the closed ball of radius $\vep$ and $\SC^0$ is the set of continuous functions $h\colon \mathbb{R} \to \mathbb{R}$ satisfying $h(0) = 0$. Equivalently, the generalized dynamical ball centered in $x \in X$ with radius $\vep>0$ can be written as 
$$\tilde{\Ga}_\vep(x) = \{y \in X : \exists h \in \SC^0 \ \text{s.t.}\ d(\phi_s(x), \phi_{h(s)}(y))\leq \vep \enspace \forall s \in \R\}.$$
\begin{definition}
A Borel probability measure $\mu$ on $X$ is called \textbf{expansive} for the flow $\phi$ if there exists $\vep > 0$ such that $\mu(\tilde\Gamma_\vep(x)) = 0$ for every $x \in X$. 
\end{definition}

We stress that, similarly to discrete-time systems, expansive measures for flows are not required to be invariant for the flow. 

The measurable notion of expansivity inherits several key properties from the topological one. In particular, expansive measures avoid singularities in their support, remain invariant under topological equivalence, and are closely related to the behavior of time-$t$ maps associated with the flow. Additionally, they exhibit natural behavior under suspensions of homeomorphisms, confirming the strong connection between the discrete and continuous-time theories.

Since the generalized dynamical ball $\tilde{\Ga}_\vep(x)$ is, by definition, an uncountable union of Borel sets, there is no guarantee that this ball is Borel measurable itself. Because of this, for a not necessarily Borel subset $A \subseteq X$, Carrasco-Oliveira and Morales considered $\mu(A)=0$ if $\mu(B)=0$ for every Borel set $B \subseteq A$. Since generalized dynamics ball plays the central role in the theory of expansive measures, the lack of measurability poses some difficulties for their study. Our first goal in this paper is to show that, under the assumption that the flow is regular, one can redefine expansive measures through another type of dynamical balls which actually are Borel sets. 

To improve the readability of this introduction, we will define here the concepts that are necessary to state our main results. More technical or standard concepts are postponed until Section \ref{section: prelim}. We write $Rep$ for the set of increasing homeomorphisms $h \colon \R \to \R$ satisfying $h(0)=0$ and $$Rep(\alpha)=\left\{h\in Rep : \left|\frac{h(s)-h(t)}{s-t} - 1\right|\leq \al \ \forall s \neq t\right\}.$$ 
Let $\phi$ be a continuous flow on $X$. For $x \in X$, $\al \in (0,1)$, and $\vep>0$ we write 
$$ \Gamma_{\vep}^\al(x)=\{y\in X:\exists h\in Rep(\al)\enspace \text{s.t.} \enspace d(\phi_s(x),\phi_{h(s)}(y))\leq\vep \enspace \forall s\in \mathbb{R}\}.$$

Our first main result guarantees that for continuous flows without fixed points using $\Ga_\vep^\vep(x)$ instead of $\tilde\Ga_\vep(x)$ does not represent any restriction. 

\begin{mainthm}\label{thm_al_vep_ball}
    Let $\phi$ be a continuous flow on a compact metric space $X$ without fixed points and $\mu$ be a Borel probability measure on $X$. Then the following statements are equivalent:
    \begin{enumerate}
        \item\label{thmA_item1} $\mu$ is expansive for $\phi$;
        \item\label{thmA_item2}  There exists $\vep>0$ such that $\mu(\Ga^\vep_\vep(x))=0$ for every $x \in X$;
        \item\label{thmA_item3} There exists $\vep>0$ such that $\mu(\Ga^\vep_\vep(x))=0$ for $\mu$-almost every $x \in X$.
    \end{enumerate} 
\end{mainthm}

 We observe that Theorem \ref{thm_al_vep_ball} is connected with the concept of $\SF$-expansive measures introduced in \cite{F_expansive}. Indeed, Theorem \ref{thm_al_vep_ball} asserts that a measure is expansive if and only if it is $\SF$-expansive for $\SF = Rep(\vep)$.

As we will see later, this equivalence allows us to deal with several important properties of expansive measure for regular flows. Moreover, using generalized dynamical ball seems to be more suitable to deal with the entropy of flows. For instance, in \cite{SV}, they propose a new definition for entropy, using generalized dynamical ball, for which positiveness is well-behaved under flow equivalence.

One of our goals is to address the problem of the existence of expansive invariant measures for flows without fixed points. An example of an expansive flow without expansive invariant measures can be obtained through the suspension of the homeomorphism presented in \cite{LMS}. 
Nevertheless, we still can recover the result from \cite{AM} to the setting of regular flows, proving expansivity of every ergodic invariant measure with positive entropy. 

For $x\in X$, $t \in \R^+$, and $\vep>0$ we write
$$\Ga^+_\vep(x) =\{y\in X: \exists h\in Rep \enspace \text{s.t.} \enspace d(\phi_s(x)),\phi_{h(s)}(y)\leq \eps \enspace \forall s\in \R^+\}.$$

\begin{definition}
A Borel probability measure $\mu$ on $X$ is called \textbf{positively expansive} for the flow $\phi$ if there exists $\vep > 0$ such that $\mu(\Gamma_\vep^+(x)) = 0$ for every $x \in X$. 
\end{definition}

\begin{mainthm}\label{entropythm}
    Let $\phi$ be a continuous flow on a compact metric space $X$ without fixed points. If $\mu$ is an ergodic $\phi$-invariant measure with positive entropy, then $\mu$ is (positively) expansive. In particular, if $\phi$ has positive topological entropy, then $\phi$ admits (positively) expansive invariant measure. 
\end{mainthm}

To prove Theorem \ref{entropythm}, we cannot straightforwardly rewrite the proof of its version for homeomorphisms presented in \cite{AM}. Indeed, the presence of reparametrizations makes the role of the results in Section \ref{sec: rep} central in our proof. 
Another crucial tool in our proof is the version of the famous Brin-Katok's Theorem for generalized Bowen balls, which can be found in \cite{JCWZ}. Here, we provide an alternative proof of this result, focusing on the topological nature of the problem and the structure of generalized dynamical balls. A key step in our argument (see Lemma~\ref{lemma: coverballs}) establishes a connection between generalized and classical dynamical balls, that we believe that may be useful in future investigations (see Section~\ref{sec: Brin-Katok}).

We also extend to flows results from \cite{LMS} describing the set of expansive measures. We show that the set of expansive measures is a $G_{\de\si}$-subset of the set of all Borel probability measures endowed with the weak*-topology. That is, the set of expansive measures for a regular flow can be written as a countable union of $G_\de$ sets (i.e. countable intersection of open sets). Additionally, we show that every expansive measure for a regular flow can be approached in the weak*-topology by expansive measures (not necessarily invariant) whose support is an invariant set for the flow. 

\begin{mainthm}\label{thm_supp_invariant_g_de_si}
    Let $X$ be a compact metric space and $\SM(X)$ be the set of Borel probability measures on $X$ endowed with the weak*-topology. If $\phi$ is a continuous flow on $X$ without fixed points, then the set of expansive measures for $\phi$ is a $G_{\de\si}$-subset of $\SM(X)$. Moreover, every expansive measure for $\phi$ is weak*-accumulated by expansive measures whose support is $\phi$-invariant.
\end{mainthm}

As an application of our results,  we can also explore the size of stable classes for ergodic measures with positive entropy, thereby extending to the setting of regular flows some of the results from \cite{AM}.

\begin{mainthm}\label{thm:classes}
    Let $\phi$ be a continuous flow on $X$. The stable classes of a continuous flow $\phi$ have zero measure with respect to any ergodic $\phi$-invariant measure with positive entropy. In particular, if $\phi$ has an ergodic invariant measure with positive entropy, then $\phi$ has uncountably many stable classes. 
\end{mainthm}

The remainder of this text is organized as follows. In Section \ref{section: prelim}, we present the basic concepts and state known results used throughout this work. Section \ref{sec: rep} is devoted to the study of the reparametrizations used in the definition of expansive measures. In particular, we provide the reader with a proof for Theorem \ref{thm_al_vep_ball}. In Section \ref{sec: Brin-Katok}, we present a proof of the Brin-Katok local entropy formula for generalized dynamical balls and prove Theorem \ref{entropythm}. In Section \ref{sec: thmE}, we prove Theorem \ref{thm_supp_invariant_g_de_si}. Finally, in Section \ref{section_thmD} we provide the reader with a proof for Theorem  \ref{thm:classes}.

\section{Preliminaries}\label{section: prelim}

In this section, we present the main definitions used in this work and provide some previously established results on the theory of expansiveness that will be instrumental in our analysis. Hereafter, $X$ denotes a compact metrizable space and $d$ a compatible metric. 
Our results do not depend on the choice of $d$. Let $\SK(X)$ be the set of all non-empty compact subsets of $X$. We endow $\SK(X)$ with the Hausdorff metric defined for $A,B \in \SK(X)$ as
$$d_H(A,B)= \max \left\{\sup_{a \in A}d(a,B), \sup_{b \in B} d(b,A)\right\}.$$

By a \textbf{continuous flow on $X$}, we mean a continuous action of $\R$ on $X$, i.e., a continuous map $\phi\colon \R\times X \to X$ satisfying $\phi(0,x)=x$ and $\phi(t+s,x)=\phi(s,\phi(t,x))$, for every $x\in X$ and $t,s\in \R$. We denote by $\phi_t=\phi(t\cdot)$ the time-$t$ map of $\phi$. We say that $x \in X$ is a \textbf{fixed point} for $\phi$ if $\phi_t(x)=x$, for every $t\in \R$. We call $\phi$ a \textbf{regular flow}, if it is a continuous flow without fixed points. The \textbf{orbit} of a point $x \in X$ under $\phi$ is the set $$O(x)=\{\phi_t(x) \in X : t\in \R\}.$$ 
Also, the positive and negative orbits of $x$ are respectively the sets
 $$ O^+(x)=\{\phi_t(x) \in X : t\geq 0\} \textrm{ and } O^-(x)=\{\phi_t(x) \in X : t\leq 0\}.$$

The next Lemma is an elementary result from flows theory that will be used several times throughout this work.  
\begin{lemma}[\cite{BW}]\label{lem: Bowen}
    Let $\phi$ be a continuous flow on $X$. For every $\vep>0$, there is $\te>0$ such that for every $x\in X$ and every $s \in [-\te,\te]$, one has $$d(\phi_{t}(x),\phi_{t+s}(x))\leq \vep,$$
    for every $t\in \R$.
\end{lemma}

Next, we define some notation that will be pivotal in our analysis. Let us first recall the classical definition of Bowen ball for flows. The \textbf{Bowen ball} or \textbf{dynamical ball} centered at $x\in X$ with radius $\delta>0$ is defined as
$$ B^{\infty}_{\phi}(x,\vep)=\{y\in X : d(\phi_t(x),\phi_t(y))\leq\vep \enspace \forall t\in \mathbb{R}\}.$$
The idea of the Bowen ball centered at $x$ is to capture all the points that remain $\eps$-close to $x$ all the time. On the other hand, expansivity is defined in terms of reparametrizations and for this reason, the classical Bowen ball is not efficient to deal with expansive flows. 

Let us now recall some generalizations of the  Bowen ball. Let $\SC(\R)$ denote the set of continuous functions from $\R$ to itself. We will fix the notation for some subsets of $\SC(\R)$ that will constantly appear in this work. We write
$$\SC^0=\{h\in \SC(\mathbb{R}): h(0)=0\};$$
$$Rep=\{h\in \SC^0: h\textrm{ in an increasing homeomorphism}\};$$
$$Rep(\al)= \left\{h \in Rep : \left|\frac{h(s)-h(t)}{s-t} -1\right| \leq \al \enspace\forall s\neq t \right\}.$$
Observe that we have the following strict inclusion $Rep(\alpha)\subset Rep\subset \SC^0$.

Next, we recall the generalized Bowen ball considered in \cite{CarrascoMorales2014}. Let $\phi$ be a continuous flow on $X$. For each $x \in X$, $t \in \R$, $\al \in (0,1)$, and $\vep>0$ the \textbf{generalized Bowen ball} is given by
$$\tilde{\Ga}_\vep(x) = \{y \in X : \exists h \in \SC^0, \enspace d(\phi_s(x),\phi_{h(s)}(y))\leq \vep \enspace \forall s \in \R\}.$$

As observed in \cite{CarrascoMorales2014}, generalized Bowen balls are not necessarily Borel sets, which can make their analysis from a measure-theoretical perspective more challenging. To overcome this issue, we shall consider the following two classes of generalized Bowen balls that will be more suitable for our approach. Let $x\in X$, $\al \in (0,1)$, and $\vep>0$. We write
$$\Gamma_{\vep}(x)=\{y\in X : \exists h\in Rep, \enspace  d(\phi_s(x),\phi_{h(s)}(y))\leq\vep \enspace\forall s\in \mathbb{R}\}.$$
The \textbf{$(\al,\vep)$-ball} centered in $x$ is defined as
$$ \Gamma_{\vep}^\al(x)=\{y\in X:\exists h\in Rep(\al),\enspace d(\phi_s(x),\phi_{h(s)}(y))\leq\vep \enspace \forall s\in \mathbb{R}\}.$$
We observe that $\Ga^\al_\vep(x)\subseteq \Gamma_{\vep}(x)\subseteq\tilde\Gamma_{\vep}(x)$. To finish the definitions of generalized Bowen balls, we define the one-sided generalized Bowen ball. The \textbf{generalized forward dynamical ball} of radius $\vep>0$ centered in $x$ is given by
$$\Gamma^+_{\eps}(x)=\{y\in X: \exists h\in Rep, \ d(\phi_s(x),\phi_{h(s)}(y))\leq\eps \enspace \forall s\in [0,\infty)\}.$$
Note that $\Gamma_{\eps}(x)\subset \Gamma^+_{\eps}(x)$, for every $x\in X$ and every $\eps>0$.

Throughout this work we consider $X$ endowed with the Borel $\si$-algebra. We write $\SM(X)$ for the set of all Borel probability measures of $X$. We always consider $\SM(X)$ endowed with the weak*-topology. We say that a set $\SS \subseteq \SM(X)$ is \textbf{convex} if for every $\mu,\nu \in \SS$ and $s \in[0,1]$ it holds that $s\mu + (1-s)\nu \in \SS$.

Let $f \colon X \to X$ be a continuous map and $\phi$ be a continuous flow on $X$. We say that $\mu\in \SM(X)$ if \textbf{$f$-invariant} if $\mu(f^{-1}(A))= \mu(A)$ for every measurable set $A \subseteq X$. We say that $\mu \in \SM(X)$ is \textbf{$\phi$-invariant} if it is $\phi_t$-invariant for every $t \in \R$. We write $\SM_{\phi}(X)$ for the set of all $\phi$-invariant measures. An $f$-invariant (resp. $\phi$-invariant) measure $\mu \in \SM(X)$ is called \textbf{ergodic} if $\mu(A)\mu(A^c)=0$ for every $f$-invariant (resp. $\phi$-invariant) measurable set $A \subseteq X$. 
It is well known that a $\phi$-invariant ergodic measure $\mu \in \SM_\phi(X)$ is ergodic for $\phi_t$ except by at most countably many $t \in \R$. 

Given a continuous map $f \colon X \to X$, for each $x \in X$, $n \in \N$ and $\vep>0$, we define the \textbf{$(n,\vep,f)$-ball} centered in $x$ as 
$$B^n_f(x,\vep)=\{y \in X : d(f^i(x),f^i(y)) \leq \vep, \ i=0, \ldots,n\}.$$ 
The Bowen ball centered in $x$ with radius $\vep>0$ is given by
$$B^\infty_f(x,\vep) = \bigcap_{n \in \N} B^n_f(x,\vep).$$ 

Analogously, given a continuous flow $\phi$ on $X$, for each $x \in X$, $t \in \R$ with $t \geq 0$ and $\vep>0$, we define the \textbf{$(t,\vep,\phi)$-ball} centered in $x$ as 
$$B^t_\phi(x,\vep)=\{y \in X : d(\phi_s(x),\phi_s(y)) \leq \vep, \ \forall s \in [0,t]\}.$$

The entropy of an ergodic $f$-invariant measure $\mu$ can be defined via Katok's formula \cite{Katok80}. For $\vep>0$ and $\de>0$, we denote by $S(n,\de,\vep,f)$ the minimal number of $(n,\vep,f)$-balls necessary to cover a measurable set $K \subseteq X$ with $\mu(K)\geq1-\de$. The \textbf{entropy} of $\mu$ with respect to $f$ is given by
    \begin{equation*}
        h_\mu(f) = \lim_{\vep\to0}\limsup_{n\to\infty} \frac{\log(S(n,\de,\vep,f))}{n}.
    \end{equation*}
The limit above is independent of $\delta$.

\begin{definition}
Let $\phi$ be a continuous flow on $X$ and $\mu \in \SM_\phi(X)$. The entropy of $\mu$ with respect to $\phi$ is given by $$h_\mu(\phi)=h_\mu(\phi_1).$$
\end{definition}

We recall the central concept considered in this work. 
\begin{definition}
    A measure $\mu \in \SM(X)$ is called \textbf{expansive} for the continuous flow $\phi$ if there exists $\eps>0$ such that $\mu(\Gamma_{\eps}(x))=0$ for every $x\in X$. We call such $\vep$ is an \textbf{expansivity constant} for $\mu$.
\end{definition}
\begin{remark}
    We emphasize that expansive measures are not required to be invariant for the flow.
\end{remark}

\section{Reparametrizations of Regular Flows}\label{sec: rep}
Our goal in this section is to show some interesting properties of reparameterizations for regular flows. In particular, we shall see that dealing with expansive measures for regular flows, one can naturally focus on the modified Bowen balls with the reparametrizations possessing some regularity.

 As previously mentioned, by definition, for every $x \in X$, $\al \in (0,1)$ and $\vep>0$ we have $\Ga^\al_\vep(x)\subseteq \Gamma_{\vep}(x)\subseteq\tilde\Gamma_{\vep}(x)$. 
In this section, we will prove that for regular flows, given $\al \in (0,1)$, if $\vep$ is sufficiently small then $\Ga_\eps(x)\subseteq\Ga^\al_\alpha(x)$ for every $x \in X$. Combined with Lemma \ref{lemma_regular_cont_homeo}, this implies that we can define expansive measures using any of the modified Bowen balls.

Due to the equivalence of the modified dynamical balls, we will strategically work mostly with the balls with reparametrization in $Rep(\al)$. Our choice of using such balls follows from the fact that, as we will see in this section, they are closed sets in particular, they are Borel sets, which is not guaranteed for the balls with reparametrizations in $\SC^0$ or $Rep$.

\begin{lemma}[\protect{\cite[Lemma 2.3]{CarrascoMorales2014}}]\label{lemma_regular_cont_homeo}
    If $\phi$ is a regular flow on $X$, then for every $\vep>0$ there exists $\de>0$ such that $\tilde{\Ga}_\de(x) \subseteq \Ga_\vep(x)$ for every $x \in X$.
\end{lemma}

One of the advantages of considering reparametrizations in $Rep(\al)$ is the fact that we can ensure convergence of reparametrizations in some sense. More precisely, we have the following lemma due to Komuro \cite{Ko}. 

\begin{lemma}[\protect{\cite[Lemma 2.2]{Ko}}]\label{lemma_unif_convergence_rep}
Let $\{h_n\}_{n=1}^\infty$ be a sequence of reparametrizations in $Rep(\al)$. There exists $h\in Rep(\al)$ such that for every $N >0$ there is an increasing sequence $\{n_k\}_{k=0}^\infty \subset \N$ such that 
    $$h_{n_k}\big|_{[-N,N]} \unif h\big|_{[-N,N]}.$$
\end{lemma}

Another advantage of working with $Rep(\al)$ is that it is complete with respect to the metric defined for $g,h \in Rep(\al)$ as 
$$D(g, h) = \sum_{n=1}^\infty \frac{1}{2^n} \cdot \frac{\sup_{s \in [-n,n]} |g(s) - h(s)|}{1 + \sup_{s \in [-n,n]} |g(s) - h(s)|}.$$

As a consequence, in Lemma \ref{lemma_convergence_dyn_balls} we obtain convergence of $(\al,\vep)$-balls.

\begin{lemma}\label{lemma_convergence_dyn_balls}
    Let $\phi$ be a regular flow on $X$, $\al \in (0,1)$, and $\vep>0$. If $\{x_k\}_{k=1}^\infty \subset X$ is a sequence converging to $x \in X$, and $\{y_k\}_{k=1}^\infty$ is a sequence converging to $y \in X$ with $y_k \in \Ga_\vep^\al(x_k)$ for every $k \in \N$, then $y \in \Ga_\vep^\al(x)$. In particular, $\Ga^\al_\vep(x)$ is closed for every $x \in X$.
\end{lemma}

\begin{proof}
    For each $\xi>0$ and $n \in \N$ we define $$\SA(\xi,n) = \{h \in Rep(\al): d(\phi_s(x), \phi_{h(s)}(y)) \leq \vep + \xi, \forall s \in [-n,n]\};$$ $$\SA(\xi) = \bigcap_{n \in \N} \SA(\xi,n).$$
    
    Note that if $h \in \bigcap_{\xi>0} \SA(\xi)$, then $d(\phi_s(x),\phi_{h(s)}( y)) \leq \vep$ for every $s \in \R$. Hence, it is enough to prove that $\bigcap_{\xi>0} \SA(\xi) \neq \emptyset$. 
    
    Observe that $\SA(\xi) \subseteq \SA(\xi')$ whenever $\xi<\xi'$. Thus, if $\SA(\xi,n)$ is closed and non-empty for every $\xi >0$ and $n \in \N$, then $\SA(\xi)$ is closed and non-empty for every $\xi>0$, and consequently, $\bigcap_{\xi>0} \SA(\xi) \neq \emptyset$. 

    Fix $\xi>0$ and $n \in \N$. First, we prove that $\SA(\xi,n)$ is closed by showing that $Rep (\al) \setminus \SA(\xi,n)$ is open. Fix $h \in Rep(\al) \setminus \SA(\xi,n)$. Then there exists $s_0 \in [-n,n]$ such that $d(\phi_{s_0}(x), \phi_{h({s_0})}( y)) > \vep + \xi$. Let $0 < \eta < d(\phi_{s_0}(x), \phi_{h(s_0)}(y)) -\vep - \xi$. By Lemma \ref{lem: Bowen}, there exists $\De>0$ such that $d(\phi_r(z), \phi_t(z))< \eta$ for every $z \in X$ and $r, t \in \R$ satisfying $|r-t|<\De$.
    
    If $g \in Rep(\al)$ and $D(g,h)$ is sufficiently small, then for every $s \in [-n,n]$ we have $|g(s)-h(s)|<\De$, which in turn implies $d(\phi_{h(s)}(y),\phi_{g(s)}(y))<\eta$. As a consequence, by the triangle inequality, we conclude that $d(\phi_{s_0}(x),\phi_{g(s_0)}(y)) > \vep + \xi$ which completes the proof that $Rep (\al) \setminus \SA(\xi,n)$ is open.

    Now, let us prove that $\SA(\xi,n) \neq \emptyset$. 
    Since $y_k \in \Ga_\vep^\al(x_k)$ for every $k \in \N$, there exists a sequence $\{h_k\}_{k = 0}^\infty \subset Rep(\al)$ such that $d(\phi_s(x_k), \phi_{h_k(s)}(y_k))\leq \vep$ for every $s \in \R$ and $k \in \N$. Let $h \in Rep(\al)$ be provided by Lemma \ref{lemma_unif_convergence_rep}. In particular, up to passing to a subsequence, we can assume that $$h_{k}\big|_{[-n,n]} \unif h\big|_{[-n,n]}.$$
    Hence, there is $k_0 \in \N$ such that for every $k\geq k_0$ one has $d(\phi_{h_k(s)}(z), \phi_{h(s)}(z))<\frac{\xi}{3}$ for every $s\in [-n,n]$, and $z \in X$. 
    
    By the continuity of the flow, since $\{x_k\}_{k=1}^\infty$ converges to $x$ and $\{y_k\}_{k=1}^\infty$ converges to $y$, there is $k_1 \geq k_0$ such that $d(\phi_s(x_k), \phi_s(x))< \frac{\xi}{3}$ and $d(\phi_{h(s)}(y_k), \phi_{h(s)}(y)) < \frac{\xi}{3}$ for every $k \geq k_1$ and $s \in [-n,n]$. 
    
    Therefore, by the triangle inequality, we conclude that for every $k \geq k_1$ and $s \in [-n,n]$ it holds that
    \begin{multline*}
        d(\phi_s(x), \phi_{h(s)}(y)) \leq d(\phi_s(x), \phi_s(x_k)) + d(\phi_s(x_k), \phi_{h_k(s)}(y_k)) \\+ d(\phi_{h_k(s)}(y_k), \phi_{h(s)}(y_k)) + d(\phi_{h(s)}(y_k),\phi_{h(s)}(y)) 
        \\ \leq \frac{\xi}{3} + \vep + \frac{\xi}{3} +\frac{\xi}{3} \leq \vep+\xi. 
    \end{multline*}

    Lastly, to see that $\Ga^\al_\vep(x)$ is closed is enough to consider $\{x_k\}_{k=1}^\infty$ as the constant sequence in $x$. This implies that if $y \in X$ is the limit of a sequence $\{y_k\}_{k=1}^\infty \subseteq \Ga^\al_\vep(x)$, then $y \in \Ga_\vep^\al(x)$.
\end{proof}

The next lemma is a combination of \cite[Lemma 3.3]{Ko} and \cite[Lemma 3.4]{Ko} adapted to our notation. 

\begin{lemma}\label{lemma_compilation_komuro}
    If $\phi$ is a regular flow on $X$, then there exists $T_0>0$ such that for every $\al \in (0,1)$ and $T \in (0,T_0)$, there exists $\vep_0>0$ such that for $x,y \in X$, if $d(\phi_s(x),\phi_{h(s)}(y))\leq \vep$ for every $s \in [0,T]$ and some $h \in Rep$ and $\vep \in (0,\vep_0]$, then $\left|\frac{h(T)}{T} -1\right|\leq\al$.
\end{lemma}

\begin{lemma}\label{lemma: reparametrization}
    Let $\phi$ be a regular flow on $X$. For every $\al \in (0,1)$ there exists $\vep_0>0$ such that if $x,y\in X$ and $h\in Rep$ satisfy $d(\phi_s(x),\phi_{h(s)}(y))\leq\vep$ for every $s\in \R$ and some $\vep \in (0,\vep_0]$, then there is $g\in Rep(\al)$ such that $d(\phi_s(x),\phi_{g(s)}(y))\leq\alpha$ for every $s \in \R$.
\end{lemma}

\begin{proof}
    Fix $\al \in (0,1)$ and let $T_0>0$ be given by Lemma \ref{lemma_compilation_komuro}. By the continuity of $\phi$, we can choose $T \in (0,T_0)$ so that $d(x,\phi_s(x))\leq \frac{\alpha}{3}$, for every $s\in [0,2T]$. Now, fix $0<\eps_0\leq \frac{\alpha}{3}$ given by Lemma \ref{lemma_compilation_komuro}. Let $x,y \in X$, let $h \in Rep$ be such that $d(\phi_s(x),\phi_{h(s)}(y))\leq \vep$ for every $s \in \R$ with $\vep \in (0,\vep_0]$. Consider the sequence $\{T_k\}_{k\in \Z}$, where $T_k=kT$. For each $k \in \Z$ we define $g_k \colon \R \to \R$ as $g_k(s)=h(T_k+s)-h(T_k)$. Note that each $g_k \in Rep$ and for every $s \in [0,T]$ it holds that
    $$d(\phi_s(\phi_{T_k}(x)), \phi_{g_k(s)}(\phi_{h(T_k)}(y))) = d(\phi_{s+T_k}(x), \phi_{h(s+T_k)}(y))\leq \vep.$$
    As a consequence, by Lemma \ref{lemma_compilation_komuro}, for every $k \in \Z$ one has $$\left|\frac{g_k(T)}{T}-1 \right|\leq\alpha.$$

We define the desired reparameterization $g \colon \R \to \R$ as 
$$
g(s)=
\begin{cases}
g_k(T), &\textrm{ if } s=T_k \\
\dfrac{g_{k+1}(T)}{T}(s-T_{k})+g_k(T), &\textrm{ if } s\in [T_k,T_{k+1}].
\end{cases}
$$

It is easy to see that $g \in Rep$. In order to prove that $g \in Rep(\al)$, choose two real numbers $t\neq r$ and assume without loss of generality that $t>r$. Then, there are $k\in \Z$ and $\ell \in \N_0$ such that $r \in [T_k,T_{k+1}]$ and $t \in [T_{k+\ell},T_{k+\ell+1}]$. Hence, we obtain
\begin{align*}
\left| g(t) - g(r) - (t - r) \right|
&= \left| g(t) - g(T_{k+\ell}) - (t - T_{k+\ell})
   + \left( \sum_{i=1}^\ell g(T_{k+i}) - g(T_{k+i-1}) \right) \right. \\
&\quad \left. - \left( \sum_{i=1}^\ell \left(T_{k+i} - T_{k+i-1}\right)\right)
   + g(T_k) - g(r) - (T_k - r) \right| \\
&\leq \left| g(t) - g(T_{k+\ell}) - (t - T_{k+\ell}) \right| \\
&\quad + \sum_{i=1}^\ell \left| g(T_{k+i}) - g(T_{k+i-1}) - (T_{k+i} - T_{k+i-1}) \right| \\
&\quad + \left| g(T_k) - g(r) - (T_k - r) \right| \\
&\leq \alpha \left( t - T_{k+\ell} + T_k - r
   + \sum_{i=1}^\ell (T_{k+i} - T_{k+i-1}) \right) \\
&= \alpha (t - r).
\end{align*}

It remains to prove that $d(\phi_s(x),\phi_{g(s)}(y))\leq\alpha$ for every $s \in \R$. Fix $s\in \R$ and let $k \in \Z$ be such that $s\in [T_k,T_{k+1}]$.
It follows from the definition of $g$ that $$|g(s)-g_k(T_k)|=\left|\dfrac{g_{k+1}(T)}{T}(s-T_{k})\right|\leq(1+\al)T\leq2T.$$
Therefore, 
\begin{align*}
    d(\phi_s(x),\phi_{g(s)}(y))&\leq d(\phi_s(x),\phi_{T_k}(x))+d(\phi_{T_k}(x),\phi_{h(T_k)}(y)) + d(\phi_{h(T_k)}(y),\phi_{g(s)}(y))\\&\leq 3\eps\leq\alpha,
\end{align*}
which completes the proof.
\end{proof}

\begin{corollary}\label{cor_equal_balls}
    If $\phi$ is a regular flow on $X$, then for every $\vep \in (0,1)$ there exists $\de_0>0$ such that for every $\de \in (0,\de_0]$ and $x \in X$ we have $\Ga_\de(x) \subseteq \Ga_\vep^\vep(x)$.
\end{corollary}

\begin{proof}[Proof of Theorem \ref{thm_al_vep_ball}]
    Trivially we have \eqref{thmA_item1}$\Rightarrow$\eqref{thmA_item2}$\Rightarrow$\eqref{thmA_item3}. We will complete the proof by showing \eqref{thmA_item3}$\Rightarrow$\eqref{thmA_item1}.  Suppose that there is $\vep \in (0,1)$  such that $\mu(\Ga^\vep_\vep(x))=0$ for $\mu$-almost every $x\in X$ but $\mu$ is not expansive. 
    
    Let $M_{\vep} \subseteq X$ be a Borel set with $\mu(M_{\vep})=1$ such that $\mu(\Ga^\vep_\vep(x))=0$ for every $x \in M_\vep$.
    By Lemma \ref{lemma_regular_cont_homeo}, it suffices to prove that there exists $\de>0$ such that for every $x \in X$ and every Borel set $B \subseteq \Ga_\de(x)$ it holds that $\mu(B)=0$. 
    
    Let $\de_0>0$ be given by Corollary \ref{cor_equal_balls}. We claim that every $\de \in (0,\de_0]$ satisfies the desired property. 
    Suppose that there exists $x_0 \in X$, $\de \in (0,\de_0]$, and some Borel set $B \subseteq \Ga_\de(x_0)$ such that $\mu(B)>0$. Then $B \cap M_\vep \neq \emptyset$, and conquently, there exists $y_0\in \Ga_{\de/2}(x_0)$ such that $\mu(\Ga^\vep_\vep(y_0))=0$. 

Let $z\in\Ga_{\de/2}(x_0)$ and let $h\in Rep$ be such that for every  $t\in\mathbb{R}$ it holds that
\begin{align*}
    d(\phi_t(x_0),\phi_{h(t)}(z))\leq \frac{\de}{2}.
\end{align*}
Since $y_0\in \Ga_{\de/2}(x_0)$, there exists $g\in Rep$ such that for every  $t\in\mathbb{R}$ it holds that
\begin{align*}
    d(\phi_t(x_0),\phi_{g(t)}(y_0))\leq \frac{\de}{2}.
\end{align*}
Hence, for every $t\in \mathbb{R}$ we have
\begin{align*}
    d(\phi_{g(t)}(y_0),\phi_{h(t)}(z))\leq d(\phi_t(x_0),\phi_{h(t)}(z))+
    d(\phi_t(x_0),\phi_{g(t)}(y_0))\leq\de.
\end{align*}

Letting $s = g(t)$, we obtain that for every $s\in\mathbb{R}$ it holds that
\begin{align*}
    d(\phi_s(y_0),\phi_{h\circ g^{-1}(s)}(z))\leq \de.
\end{align*}
Since $h\circ g^{-1}\in Rep$, it follows that $z\in \Ga_{\de}(y_0)$. Therefore, $B \subseteq \Ga_{\de/2}(x_0) \subseteq \Ga^\vep_\vep(y_0)$ and $0<\mu(B)\leq \mu(\Ga^\vep_\vep(y_0))=0$, which is a contradiction.
\end{proof}

Given $x \in X$, $\al \in (0,1)$ and $\vep>0$, we can also define the generalized forward $(\al,\vep)$-ball with reparametrization in $Rep(\al)$ as
$$\Gamma^{\al+}_{\eps}(x)=\{y\in X: \exists h\in Rep(\al), \ d(\phi_s(x),\phi_{h(s)}(y))\leq\eps \enspace \forall s\in [0,\infty)\}.$$
We observe that using the same arguments we obtain a version of Theorem \ref{thm_al_vep_ball} for generalized forward balls as follows.

\begin{theorem}\label{thmA_forward}
        Let $\phi$ be a regular flow on $X$ and $\mu \in \SM(X)$. Then the following statements are equivalent:
        \begin{enumerate}
        \item $\mu$ is positively expansive;
        \item  There exists $\vep>0$ such that $\mu(\Ga^{\eps+}_\vep(x))=0$ for every $x \in X$;
        \item  There exists $\vep>0$ such that $\mu(\Ga^{\eps+}_\vep(x))=0$ for $\mu$-almost every $x \in X$.
        \end{enumerate}
\end{theorem}

\begin{remark}
Note that the assumption that $\phi$ is a regular flow is crucial for the arguments in this section. Indeed, in presence of singularities, a regular point may take long time to leave a small neighborhood of a singularity and this poses substantial obstacles for proving essential lemmas, including Lemma \ref{lemma_regular_cont_homeo}, which does not hold in the presence of singularities. In particular, the results in this section and the subsequent ones may not hold in presence of singularities. 
\end{remark}

\section{Brin-Katok Formula and Measures with Positive Entropy}\label{sec: Brin-Katok}
Our goal in this section is to prove a version of the Brin-Katok local entropy formula for regular flows using generalized dynamical balls and prove Theorem \ref{entropythm}. We first recall the original result for discrete-time systems.

\begin{theorem}[\cite{BK}]\label{thm_brin_katok}
    Let $f \colon X \to X$ be a continuous map. If $\mu$ is an ergodic $f$-invariant measure with $h_\mu(f)<\infty$, then for $\mu$-almost every $x\in X$ it holds that
     $$\lim_{\vep \to 0} \liminf_{n \to \infty} \frac{-\log(\mu(B_f^n(x,\eps)))}{n}= \lim_{\vep \to 0} \limsup_{n \to \infty} \frac{-\log(\mu(B_f^n(x,\eps)))}{n} = h_\mu(f).$$
\end{theorem}

For $x\in X$, $t \in \R^+$, and $\vep>0$ we write
$$\Gamma^+_{\eps}(x,t)=\{y\in X: \exists h\in Rep \enspace \text{s.t.} \enspace d(\phi_s(x)),\phi_{h(s)}(y)\leq \eps \enspace \forall s\in [0,t]\}.$$

\begin{theorem}\label{thm: BK-formula}
    Let $\phi$ be a regular flow on $X$. If $\mu$ is an ergodic $\phi$-invariant measure with $h_\mu(\phi)< \infty$, then for $\mu$-almost every $x \in X$ we have 
    $$\lim_{\vep \to 0} \liminf_{t \to \infty} \frac{-\log(\mu(\Ga_\vep^+(x,t)))}{t}= \lim_{\vep \to 0} \limsup_{t \to \infty} \frac{-\log(\mu(\Ga_\vep^+(x,t)))}{t} = h_\mu(\phi).$$
\end{theorem}

In order to prove Theorem \ref{thm: BK-formula}, we will need some lemmas. First, we observe that using a similar argument as in Lemma \ref{lemma: reparametrization} we obtain a version of Corollary \ref{cor_equal_balls} for forward dynamical balls with finite time as below.

\begin{lemma}\label{lemma_eq_balls_finita}
    If $\phi$ is a regular flow on $X$, then for every $\al \in (0,1)$ there exists $\de_0>0$ such that for every $\de \in (0,\de_0]$ and $x \in X$ we have $\Ga_\de^+(x,t) \subseteq \Ga_\al^{\al+}(x,t)$.
\end{lemma}

Informally, the next lemma says that we can cover a generalized Bowen balls with a controlled amount of discrete Bowen balls whose centers belong to the
initial generalized Bowen ball.
\begin{lemma}\label{lemma: coverballs}
    Let $\phi$ be a regular flow on $X$ and $L \in \N$. For every $\vep>0$ there is $\de_0>0$ such that for every $x \in X$, $\de \in(0,\de_0)$, and $t>L$ there exists $\{x_1, \ldots, x_\ell\} \subset\Ga_\de^+(x,t)$ such that $$\Ga_\de^+(x,t) \subseteq \bigcup_{i=1}^\ell B^{nL}_f(x_i,\vep),$$ where $n=\floor{\frac{t}{L}}$ and 
    $\ell \leq 3^{n-1}$, and $f = \phi_1$. 
    
    Furthermore, for any compact set $K\subseteq X$ containing $x$, the set $K\cap \Gamma^+_{\delta}(x,t)$ can be covered by at most $3^{n-1}$ balls $B^{nL}_f(x_i,\vep)$ whose centers belong to $K$.
\end{lemma}

\begin{proof}
    Fix $\vep>0$ and $L \in \N$. We first assume that $t = nL$ for some $n\in\N$. By Lemma \ref{lem: Bowen}, there exists $\te>0$ such that $d(\phi_s(x),x) \leq \frac{\vep}{6}$ for every $x \in X$ and $s \in [-\te,\te]$. 
    Let $\de_0>0$ be given by Lemma \ref{lemma_eq_balls_finita} for $\al = \min\{\frac{\te}{3L}, \frac{\vep}{6}\}$. 

    For each $g \in Rep(\al)$, we denote $\ga_g=(\ga_g(0), \ldots, \ga_g(k)) \in \Z^n$, where $\ga_g(k)$ is defined for each $k \in \{0, \ldots, n-1\}$ as $$\ga_g(k) = \floor{\frac{g(kL)-kL}{\al L}}.$$

    Note that for every $g \in Rep(\al)$ we have $\ga_g(0) = 0$. Moreover, for each $k \in \{0,\ldots, n-2\}$ we have $$\ga_g(k+1) - \ga_g(k) \in \{-1,0,1\}.$$ Hence, there are $3^{n-1}$ sequences $\ga_1, \ldots,\ga_{3^{n-1}} \in \Z^n$ such that for every $g \in Rep(\al)$ one has $\ga_g \in \{\ga_1, \ldots,\ga_{3^{n-1}}\} \subset \Z^n$. Set $\SS = \{\ga_1, \ldots,\ga_{3^{n-1}}\}$.

    Fix $x \in X$ and $\de \in (0,\de_0]$. By Lemma \ref{lemma_eq_balls_finita}, if $y,z \in \Ga_\de^+(x,t)$, then there exist $g,h \in Rep(\al)$ such that for every $s \in [0,t]$ one has $$d(\phi_{g(s)}(x),\phi_s(y))\leq \al \ \ \text{and} \ \ d(\phi_{h(s)}(x),\phi_s(z))\leq \al.$$

    We observe that if $\ga_g = \ga_h$, then for every $s \in [0,t]$ we have
    \begin{align*}
        |g(s)-h(s)| & \leq \left|(g(s) - s)-\left(g(\floor{\frac{s}{L}}L) -\floor{\frac{s}{L}}L\right)\right| \\
        & + \left|\left(g(\floor{\frac{s}{L}}L) -\floor{\frac{s}{L}}L\right) - \left(h(\floor{\frac{s}{L}}L) -\floor{\frac{s}{L}}L\right)\right| \\
        &+ \left|(h(s)-s) - \left(h(\floor{\frac{s}{L}}L) -\floor{\frac{s}{L}}L\right)\right| \\
        & \leq \al L + \al L\left| \frac{g(\floor{\frac{s}{L}}L) -\floor{\frac{s}{L}}L}{\al L} - \frac{h(\floor{\frac{s}{L}}L) -\floor{\frac{s}{L}}L}{\al L}\right| + \al L \\ &\leq 3\al  L \leq \te  .
    \end{align*}
    As a consequence, we obtain that $ d(\phi_{g(s)}(x), \phi_{h(s)}(x)) \leq \frac{\vep}{6}$ for every $s \in [0,t]$, which implies that
    \begin{align*}
        d(\phi_s(y),\phi_s(z))  \leq d(\phi_{g(s)}(x),\phi_s(y)) + d(\phi_{g(s)}(x), \phi_{h(s)}(x)) + d(\phi_{h(s)}(x),\phi_s(z)) <\frac{\vep}{2}.
    \end{align*}
    That is, $y \in B^t_\phi(z,\frac{\vep}{2})$ and $z \in B^t_\phi(y,\frac{\vep}{2})$.
    
    Since $|\SS|=3^{n-1}$, there exists $\{x_1, \ldots, x_\ell\} \subseteq \Ga_\de^+(x,t)$ with $\ell \leq 3^{n-1}$ such that $$\Ga_\de^+(x,t) \subseteq \bigcup_{i=1}^\ell B^t_ \phi(x_i,\frac{\vep}{2}) \subseteq\bigcup_{i=1}^\ell B^{nL}_f(x_i,\frac{\vep}{2})\subset \bigcup_{i=1}^\ell B^{nL}_f(x_i,\vep).$$

    We now prove for an arbitrary $t\in\R^+$. Let $n = \floor{\frac{t}{L}}$. Note that $t \in \left[nL, \left(n+1\right)L\right)$. We already proved that there exists $\{x_1, \ldots, x_\ell\}\subset \Ga_\de^+(x,t)$ for some $\ell \leq 3^{n-1}$, such that $\Ga_\de^+(x,nL)\subseteq\bigcup_{i=1}^\ell B^{nL}_f(x_i,\vep)$.
    Since  $\Ga_\de^+(x,t) \subseteq \Ga_\de^+(x,nL)$, we conclude that $\Ga_\de^+(x,t) \subseteq\bigcup_{i=1}^\ell B^{nL}_f(x_i,\vep)$.

    Next, we prove the furthermore part. Let $K\subseteq X$ be a compact set containing $x$. Note that $$K \cap \Ga_\de^+(x,t)\subseteq \Ga_\de^+(x,t)\subseteq \bigcup_{i=1}^\ell B^{nL}_f(x_i,\frac{\vep}{2}).$$
    In particular, for each $y\in K\cap \Ga_\de^+(x,t)$ there is $i\in \{1,\ldots,\ell\}$ such that $$y\in B^{nL}_f(x_i,\frac{\vep}{2}) \subseteq B^{nL}_f(y,\vep).$$ 
    
    Therefore, there is $\{y_1,\ldots,y_{\ell'}\}\subset K\cap \Ga_\de^+(x,t)$, with $\ell'\leq \ell \leq 3^{n-1}$, such that  \[K\cap \Ga_\de^+(x,t)\subset \bigcup_{i=1}^{\ell'} B^{nL}_f (y_i,\vep). \qedhere\]
\end{proof}

\begin{lemma}\label{lemma_calculations_size_balls}
	Let $f \colon X \to X$ be a continuous map and $\mu$ be a non-atomic $f$-invariant measure with $h_\mu(f)<\infty$. For every $\rho>0$ and $\xi \in (0,h_\mu(f))$ there exist $\vep >0$, $n_0\in \N$, and a compact set $K\subseteq X$ with $\mu(K)>1-\rho$ such that for every $x \in K$ and $n \geq n_0$ it holds that
	$$\mu(B^n_f(x,\vep)) \leq e^{-n(h_\mu(f) - \xi)}.$$
\end{lemma}

\begin{proof}
	For every $x \in X$ and $n \in \N$ and $\vep>0$ we define
	$$\SF_{n,\vep}(x) =\frac{-\log\left( \mu \left(B^n_f\left(x, \vep \right)\right)\right)}{n};$$
	$$\SF_{\vep}(x) = \liminf_{n \to \infty} \SF_{n,\vep}(x);$$ $$\SF(x)= \lim_{\vep \to 0}\SF_\vep(x).$$
	
	 By Theorem \ref{thm_brin_katok} there exists a Borel set $H \subseteq X$ with $\mu(H)=1$ such that $\SF(x) = h_\mu(f)$ for every $x \in H$. Fix $\rho>0$ and $\xi \in (0,h_\mu(f))$. Let $K'\subseteq H$ be a compact set with $\mu(K')>1-\rho$. By compactness, there exists $\vep_0>0$ such that $\SF_\vep(x) > h_\mu(f)-\xi$ for every $x \in K'$ and every $\vep \in (0,\vep_0)$.
	
	Fix $\vep \in (0,\vep_0)$. Note that $\SF_\vep$ and each $\SF_{n,\vep}$ are measurable functions on $X$.
	By Lusin's Theorem, we obtain a compact set $K \subseteq K'$ with $\mu(K)>1-\rho$ such that $\SF_\vep$ and each $\SF_{n,\vep}$ are continuous on $K$.
	In particular, by the continuity of $\SF_\vep$, there exists $n_0 \in \N$ such that for every $x \in K$ and $n \geq n_0$ it holds that
	\begin{equation}\label{eq_brin_katok}
		\frac{-\log\left( \mu (B^n_f\left(x,\vep\right))\right)}{n} \geq h_\mu(f) - \xi.
	\end{equation}
	It follows from \eqref{eq_brin_katok} that for every $x \in K$ and $n \geq n_0$ one has
	\[\mu(B^n_f(x,\vep)) \leq e^{-n(h_\mu(f) - \xi)}. \qedhere\]
\end{proof}

Now we are ready to prove Theorem \ref{thm: BK-formula} 

\begin{proof}[Proof of Theorem \ref{thm: BK-formula}]
		Let $\mu \in \SM_\phi(X)$. Recall that if $\mu$ is ergodic for $\phi$, then $\mu$ is ergodic for $\phi_t$ except by at most countably many $t \in \R$. 
		For simplicity, we assume without loss of generality that $\mu$ is ergodic for $\phi_1$. We write $f = \phi_1$, that is, $$0<h_\mu(f)=h_\mu(\phi)<\infty.$$
		
		By the continuity of the flow, for every $n \in \N$, $t \in [n,n+1)$ and $\vep>0$ there exists $0<\de<\vep$ such that 
		$B^n_f(x,\de) \subseteq B^t_\phi(x,\vep)$. 
		Hence,
		$$\lim_{\vep \to 0} \limsup_{t\to \infty} \frac{-\log\left(\mu(B^t_\phi(x,\vep)\right)}{t} \leq \lim_{\de \to 0} \lim_{n \to \infty} \frac{- \log\left(\mu(B^n_f(x,\de))\right)}{n} = h_\mu(f)= h_\mu(\phi).$$
		Since $B^t_\phi(x,\vep) \subseteq \Ga_\vep^+(x,t)$, for every $x \in X$, $\vep>0$, and $t \in \R^+$, we conclude that 
		$$\lim_{\vep \to 0} \limsup_{t\to \infty} \frac{-\log\left(\mu(\Ga_\vep^+(x,t)\right)}{t} \leq \lim_{\vep \to 0} \limsup_{t\to \infty} \frac{-\log\left(\mu(B^t_\phi(x,\vep)\right)}{t} \leq h_\mu(\phi).$$
		
		To complete the proof we will show that for every $\xi \in (0,h_\mu(\phi))$ and $\mu$-almost every $x \in X$ it holds that
		\begin{equation*}\label{eq_liminf_BK}
			\lim_{\de \to 0} \liminf_{t\to \infty} \frac{-\log\left(\mu(\Ga_\de^+(x,t)\right)}{t} \geq h_\mu(f)-\xi.
		\end{equation*}
		
		Let $H \subseteq X$ be the set of full measure provided by Theorem \ref{thm_brin_katok}. It suffices to prove that for every $x \in H$,  $\xi\in \left(0,\frac{h_\mu(\phi)}{2}\right)$, and sufficiently small $\de>0$ it holds that
			$$\liminf_{t\to \infty} \frac{-\log\left(\mu(\Ga_\de^+(x,t)\right)}{t} \geq h_\mu(f)-2\xi.$$

		Fix $\xi \in \left(0,\frac{h_\mu(f)}{2}\right)$.
		Let $\vep>0$, $L \in \N$ be such that $\frac{\log3}{L}< \xi$, and  $\de_0$ be provided by Lemma \ref{lemma: coverballs}. For every $\rho \in (0,1)$ there exists a compact set $K^\rho \subset H$ containing $x$ with $\mu(K^\rho)>1-\rho$. Let $n_\rho \in \N$ be provided by Lemma \ref{lemma_calculations_size_balls}. That is, for every $y \in K^\rho$ and $n \geq n_\rho$, it holds that 
			\begin{equation}\label{eq_BK_1}
				\mu(B^n_f(y,\vep)) \leq e^{-n(h_\mu(f) - \xi)}.
			\end{equation}
			
		By Lemma \ref{lemma: coverballs}, for every $\de \in (0,\de_0)$ and $t\geq L$, there is $A_t=\{x_1,\ldots ,x_{\ell_t}\}\subset K^\rho$ such that
		\begin{equation}\label{eq_BK_2}
		    K^\rho_\de(x,t) \subseteq \bigcup_{i=1}^{\ell_t} B^{n_tL}_f(x_i,\eps),
		\end{equation}
		where, $K^\rho_\de(x,t)=K^\rho\cap \Gamma^+_{\delta}(x,t)$, $n_t=\floor{\frac{t}{L}}$ and $\ell_t\leq 3^{n_t-1}$. 
			
			Hence, it follows from \eqref{eq_BK_1} and \eqref{eq_BK_2} that for every $t \in \R^+$ such that $n_tL \geq n_\rho$ one has
			\begin{equation}\label{eq_ball_rho}
				\mu(\Ga_\de^+(x,t)) \leq \mu(
				K^\rho_\de(x,t)) + \rho \leq 3^{n_t-1} \cdot  e^{-n_tL(h_\mu(f) - \xi)} + \rho.
			\end{equation}
			
		Since \eqref{eq_ball_rho} holds for $\rho \in (0,1)$ and large enough $t$, we conclude that 
		\begin{align*}
			\begin{split}
				\liminf_{t\to \infty} \frac{-\log\left(\mu(\Ga_\de^+(x,t)\right)}{t}  &\geq \liminf_{t\to \infty}   \frac{-\log (3^{n_t-1} \cdot e^{-n_tL(h_\mu(f) - \xi)})}{(n_t+1)L} \\ & \geq  \liminf_{t\to \infty}  \frac{n_t}{n_t+1}(h_\mu(f) - \xi)-\frac{\log 3}{L} \\ & \geq  \liminf_{t\to \infty} \frac{n_t}{n_t+1}(h_\mu(f) - \xi)-\xi \\
				&= h_{\mu}(f)-2\xi. \qedhere
			\end{split}
		\end{align*}
	\end{proof}
    
 Let us now concentrate ourselves on the proof of Theorem \ref{entropythm}.  Let us begin with following lemma that will play a crucial role in our proof. 

\begin{lemma}\label{lemma_Avep_borel}
    Let $\phi$ be a regular flow on $X$, $\mu \in \SM(X)$ and $C\geq 0$. There exists $\vep_0>0$ such that if $\mu$ is expansive for $\phi$ and $\vep \in (0,\vep_0]$ is an expansivity constant for $\mu$, then $$A_{\vep,C}=\{x \in X : \mu(\Ga_\vep^\vep(x)) \geq C\}$$ is a Borel set. 
    In particular, if $\mu$ is not expansive, then for every $\vep \in (0,\vep_0]$ there exists $C=C(\vep)>0$ such that $\mu(A_{\vep,C})>0$.
\end{lemma}

\begin{proof}
    Let $\vep_0>0$ be given by Corollary \ref{cor_equal_balls}. Then $\Ga_\vep^\vep(x)$ is a compact set for every $x \in X$ and $\vep \in (0,\vep_0]$. If we fix $\vep \in (0,\vep_0]$, then $\Ga_\vep^\vep(\cdot)$ defines a map from $X$ to $\SK(X)$, which by Lemma \ref{lemma_convergence_dyn_balls}, is upper semicontinuous. Therefore, the function $\SF_\vep \colon X \to \R$ defined as $\SF_\vep(x) = \mu(\Ga_\vep^\vep(x))$ is a measurable function. To see that $A_{\vep,C}$ is a Borel set, it is enough to observe that $A_{\vep,C} = \SF_\vep^{-1}([C,\infty))$.

    Now assume that $\mu$ is not expansive for $\phi$. By Theorem \ref{thm_al_vep_ball}, for every $\vep \in (0,\vep_0]$ there exists a Borel set $A \subseteq X$ with $\mu(A)>0$ such that $\mu(\Ga_\vep^\vep(x))>0$ for every $x \in A$. By Lusin's Theorem, there exists a compact set $K \subseteq A$ with $\mu(K) > 0$ such that $\SF_\vep$ is continuous on $K$, and consequently, there exists $C_\vep>0$ such that $\SF_\vep(x) \geq C_\vep$ for every $x \in K$. Therefore, $\mu(A_{\vep,C_\vep})>0$.
\end{proof}
Next, we are ready to prove Theorem \ref{entropythm}

\begin{proof}[Proof of Theorem \ref{entropythm}]
    Let $\mu \in \SM_\phi(X)$ be an ergodic measure such that $h_\mu(\phi)>0$. Suppose $\mu$ is not expansive. By Lemma \ref{lemma_Avep_borel}, for every sufficiently small $\vep>0$ there exists $C_\vep>0$ such that $\mu(A_{\vep})>0$, where $$A_{\vep} = \{x \in X : \mu(\Ga_\vep^\vep(x)) \geq C_\vep\}.$$ 
    
    Note that $\Ga_\vep^\vep(x) \subseteq \Ga_\vep^+(x,t)$ for every $x \in X$ and $t \geq 0$. Hence, $\mu(\Ga_\vep^+(x,t)) \geq C_\vep$ for every $t \geq 0$ and $x \in A_\vep$.
    
    Since $\mu$ is ergodic for $\phi$, $\mu$ is ergodic for $\phi_t$ except by at most countably many $t \in \R$.  
    We can assume without loss of generality that $\mu$ is ergodic for $\phi_1$. Moreover, since we assumed $h_\mu(\phi)>0$, by Theorem \ref{thm: BK-formula}, there exists a Borel set $H\subset X$ with $\mu(H)=1$ such that for every $x \in H$ it holds that $$h_{\mu}(\phi)=\lim\limits_{\eps\to 0}\limsup\limits_{t\to \infty}\frac{-\log(\mu(\Gamma^+_{\eps}(x,t)))}{t}.$$
     
    Since $\mu(A_{\eps})>0$ for every sufficiently small $\vep>0$, it follows that $H\cap A_\vep \neq \emptyset$. If we take $x_0 \in H \cap A_\vep$, then $-\log(\mu(\Gamma^{+}_{\eps}(x_0,t))\leq -\log(C_\vep)$ for every $t>0$. But this implies that $$\limsup\limits_{t\to \infty}\frac{-\log(\mu(\Gamma^+_{\eps}(x_0,t))}{t}\leq \limsup\limits_{t\to\infty }\frac{-\log(C_\vep)}{t}=0.$$
    By taking $\eps\to 0$, we obtain $$h_{\mu}(\phi)=\lim\limits_{\eps\to 0}\limsup\limits_{t\to \infty}\frac{-\log(\mu(\Gamma^+_{\eps}(x,t)))}{t}=0,$$
    contradicting the assumption $h_{\mu}(\phi)>0$, and therefore, $\mu$ must be expansive.
\end{proof}

\section{Approximating Expansive Measures by Measures with Invariant Support }\label{sec: thmE}

In this section we prove Theorem \ref{thm_supp_invariant_g_de_si}. While the proof is based on the arguments presented in \cite{LMS},  we stress that it is not a straightforward transcription of their argument and we need to be cautious. In the flows case, due to the necessity to deal with reparametrizations, several obstacles appear in the adaptation of the proofs. To overcome these obstacles, supported by Theorem \ref{thm_al_vep_ball}, we will strategically restrict ourselves to reparametrizations in $Rep(\al)$.

\begin{definition}
    Let $\phi$ be a continuous flow on $X$. We say that $\mu \in \SM(X)$ is \textbf{$(\al,\vep)$-expansive} for $\phi$ if $\mu(\Ga_\vep^\al(x))=0$ for every $x \in X$. We denote by $\ExpF(\al,\vep) \subset \SM(X)$ the set of all $(\al,\vep)$-expansive measures.
\end{definition}

We start by proving an analogue to \cite[Lemma 3.1]{LMS}. More precisely, we will show that $\ExpF(\al,\vep)$ is a $G_\de$-subset of $\SM(X)$ for every $\al \in (0,1)$ and $\vep>0$. The proof of \cite[Lemma 3.1]{LMS} strongly relies on convergence of discrete-time dynamical balls. This convergence for flows is provided by Lemma \ref{lemma_convergence_dyn_balls}.

\begin{lemma}\label{lemma_exp_gdelta}
    If $\phi$ is a regular flow on a compact metric space $X$, then $\SE^\phi(\al,\vep)$ is a $G_\de$-subset of $\SM(X)$ for every $\al \in (0,1)$ and $\vep>0$.
\end{lemma}
\begin{proof}
    Let $\al \in (0,1)$ and $\vep>0$. Set
    $$C(\rho) = \left\{ \mu \in \SM(X) : \mu(\Ga_\vep^\al(x)) \geq \rho \text{ for some } x \in X\right\}.$$

    Observe that $$\SE^\phi(\al,\ve) = \bigcap_{m\in\N} \SM(X) \setminus C(m^{-1}).$$
    
    We claim that $C(\rho)$ is closed in $\SM(X)$ for every $\rho>0$. This implies that $\SM(X) \setminus C(m^{-1})$ is open for every $m \in \N$, and consequently, that $\SE^\phi(\al,\ve)$ is a $G_\de$-subset of $\SM(X)$.
    
    Let us prove the claim. Fix $\rho >0$. Let $\{\mu_n\}_{n=1}^\infty \subset C(\rho)$ be a sequence converging to some $\mu \in \SM(X)$. Since $\mu_n \in C(\rho)$, there exists a sequence $\{x_n\}_{n=1}^\infty \subset X$ such that for every $n \in \N$ it holds that
    $\mu_n(\Ga_\vep^\al(x_n)) \geq \rho.$

    Since $X$ is compact, passing, if necessary, to a subsequence, we can assume that $\{x_n\}_{n=1}^\infty$ converges to some $x \in X$.
    Fix a compact neighborhood $K$ of $\Ga_\vep^\al(x)$. Let $O = \operatorname{int}(K)$ denote the interior of $K$. 

    Let us prove that $\Ga_\vep^\al(x_n) \subset K$ for every $n$ sufficiently large. Otherwise, there exists an increasing sequence $\{n_k\}_{k=1}^\infty \subset \N$ such that $\Ga_\vep^\al(x_{n_k}) \not\subset O$ for every $k \in \N$. Thus, there exists a sequence $\{z_k\}_{k=1}^\infty$ such that $z_k \in \Ga_\vep^\al(x_{n_k})\setminus O$ for each $k \in \N$.

    By compactness of $X$, we can assume that $\{z_k\}_{k=1}^\infty$ converges to some $z \in X$. By Lemma \ref{lemma_convergence_dyn_balls}, we have $z\in \Ga_\vep^\al(x)$. However, since $O$ is open, it follows by construction that $z \notin O$, which is an absurd. 
    
    Hence, for every sufficiently large $n \in \N$ we have $\Ga_\vep^\al(x_n) \subseteq K$, and consequently, $\mu_n(\Ga_\vep^\al(x_n)) \leq \mu_n(K)$. Since $\mu_n \to \mu$, by the Portmanteau Theorem we obtain
    $$\rho \leq  \limsup_{n \to \infty}  \mu_n(\Ga_\vep^\al(x_n)) \leq \limsup_{n \to \infty} \mu_n(K) \leq \mu(K).$$

    Therefore, $\mu(K) \geq \rho$ for every compact neighborhood of $\Ga_\vep^\al(x)$, which implies $\mu(\Ga_\vep^\al(x)) \geq \rho$.
\end{proof}

In the discrete-time setting, the proof that expansive measures can be weak*-accumulated by expansive measures whose support is invariant for the systems relies on invariance of expansive measures with respect to the pushforward preserving the same expansivity constant. That is, if $\mu$ is an expansive measure for a homeomorphism $f \colon X \to X$ with expansivity constant $\vep>0$, then $f*\mu$ is an expansive measure for $f$ with expansivity constant $\vep$. This follows directly from the fact that $f^{-1}(B_f^\infty(x,\vep)) = B_f^\infty(f^{-1}(x),\vep)$ for every $x \in X$.

For flows, due to the reparametrizations, we cannot guarantee that $\phi_t(\Ga^\al_\vep(x)) = \Ga^\al_\vep(\phi_t(x))$. In particular, we cannot guarantee that $\ExpF(\al,\vep)$ is invariant with respect to the pushforward with respect to time-$t$ maps, which is a key step of the argument. 

It turns out that we can prove such invariance for a slightly different set of expansive measures. Note that $\ExpF(\vep,\vep) \subseteq \ExpF(\de,\de)$ for every $0<\de<\vep$. We then define $$\CExpF (\vep) = \bigcap_{0 < \de<\vep} \ExpF(\de,\de).$$
It is yet unclear if $\CExpF(\vep)=\ExpF(\vep,\vep)$.

\begin{remark}\label{remark_gdelta}
    It follows from Lemma \ref{lemma_exp_gdelta} that $\CExpF(\vep)$ is a $G_\de$-subset of $\SM(X)$ for every $\vep \in (0,1)$.
\end{remark}

\begin{lemma}\label{lemma_translation_bowen_ball}
    Let $x,y \in X$ and $\vep>0$. If $d(\phi_s(x), \phi_{h(s)}(y)) \leq \vep$ for every $s \in \R$ and some $h \in Rep(\al)$, then for every $t \in \R$ we have
    $$\phi_t(y) \in \Ga^\al_\vep(\phi_{h^{-1}(t)}(x)).$$
\end{lemma}

\begin{proof}
    Let $x,y \in X$, $\vep>0$, and $h \in Rep(\al)$ such that $d(\phi_s(x), \phi_{h(s)}(y)) \leq \vep$ for every $s \in \R$. Fix $t \in \R$ and set $g(s) = h(s+h^{-1}(t)) - t$. Note that $g \in Rep(\al)$. Also, for every $s \in \R$ we have 
    $$\phi_{g(s)}(\phi_t(y))=\phi_{h(s+h^{-1}(t))-t)}(\phi_t(y)) = \phi_{h(s+h^{-1}(t))}(y).$$

    Thus, it follows that for every $s \in \R$ one has
    \[d(\phi_s(\phi_{h^{-1}}(x)), \phi_{g(s)}(\phi_t(y)) =d(\phi_{s + h^{-1}(t)}(x), \phi_{h(s+h^{-1}(t))}(y))\leq \vep. \qedhere \]
\end{proof}

Our argument to prove that $\CExpF(\vep)$ is $\phi_t*$-invariant for every $t \in \R$ relies on making small perturbations on the reparametrizations. It is not hard to see that if $h \in Rep$ and $t \in \R\setminus\{0\}$, then for every interval $I \subset \R$ containg $t$ and $s \in I \setminus\{t\}$ we can find $g \in Rep$ such that $g(t)=s$ and $h\big|_{\R \setminus I} = g\big|_{\R \setminus I}$.
We observe that if $h \in Rep(\vep)$, then it is not true that we can make a similar perturbation and obtain $g \in Rep(\vep)$. 
For example, assume that for some $h \in Rep(\vep)$ there exists $r < t < s$ such that 
$$\left| \frac{h(r)-h(t)}{r-t}-1 \right|=\left| \frac{h(s)-h(t)}{s-t}-1 \right|=\vep.$$
If $I \subset (r,s)$ and $g \in Rep$ satisfies $g(r)=h(r)$ and $g(s) = h(s)$, then it is not hard to see that $g \in Rep(\vep)$ if and only if $g(t)=h(t)$. Nevertheless, we can prove the following key lemma.

\begin{lemma}\label{lemma_rep_rational}
    Let $\phi$ be a regular flow on $X$, $x \in X$, and $\vep \in (0,1)$. If $y \in \Ga_\de^\de(x)$, then for every $\de \in (0,\vep)$ and $t \in \R$ there exists $g \in Rep(\vep)$ such that 
    \begin{enumerate}
        \item $g^{-1}(t) \in \Q$;
        \item $d(\phi_s(x), \phi_{g(s)}(y)) \leq \vep$ for every $s \in \R$.
    \end{enumerate}
\end{lemma}

\begin{proof}
    Fix $x \in X$, $t \in \R$, and $\vep \in (0,1)$. If $y \in \Ga_\de^\de(x)$, then there exists $h \in Rep(\de)$ such that $d(\phi_s(x),\phi_{h(s)}(y)) \leq \de$ for every $s \in \R$. 
    
    By Lemma \ref{lem: Bowen}, there exists $\De >0$ such that for every $z \in X$ and every $r,s \in \R$ satisfying $|r-s|<\De$, we have $d(\phi_r(z), \phi_s(z))<\vep - \de$. We can assume without loss of generality that $\De \leq \min\{\frac{\vep}{3}, \vep-\de\}$. Let $I\subset \R$ be an interval containing $t$ so that $|h(I)|<\De$. We can find $g \in Rep$ (not necessarily in $Rep(\vep)$) such that $h\big|_{\R \setminus I} = g\big|_{\R \setminus I}$ and $g^{-1}(t) \in \Q$.

    Note that, by construction, one has $g(I) = h(I)$. Hence, for every $s \in I$ we have $|g(s) - h(s)|<\De$, which implies $d(\phi_{g(s)}(y), \phi_{h(s)}(y)) < \vep -\de$. As a consequence, by the triangle inequality, we obtain that for every $s \in I$ it holds that
    \begin{equation*}
        d(\phi_s(x),\phi_{g(s)}(y)) \leq d(\phi_s(x),\phi_{h(s)}(y)) + d(\phi_{h(s)}(y),\phi_{g(s)}(y)) \leq \vep.
    \end{equation*}
    Therefore, $d(\phi_s(x),\phi_{g(s)}(y)) \leq \vep$ for every $s \in \R$.
    
    It remains to prove that $g \in Rep(\vep)$. We will show that for every $r,t \in \R$ one has
    \begin{equation}\label{eq_rep_vep}
        |g(t)-g(r) - (t-r)|\leq\vep.
    \end{equation}
    
    Since $h \in Rep(\de)$ and $h\big|_{\R \setminus I} = g\big|_{\R \setminus I}$, if $r,t \in \R\setminus I$, then \eqref{eq_rep_vep} trivially holds. 
    On the other hand, if $r,t \in I$, then, by construction, we have
    $$|g(t)-g(r) - (t-r)| \leq |g(t)-h(t)| + |g(r)-h(r)| + |t-r| \leq 3\De \leq \vep.$$
    
    Lastly, if $r \in \R \setminus I$ and $t \in I$, then
    \begin{align*}
        |g(t)-g(r) - (t-r)| &\leq |g(t) - h(t)| + |h(t)-g(r) -(t-r)|\\
        &=|g(t) - h(t)| + |h(t)-h(r) -(t-r)|\\
        &\leq \De + \de < \vep. \qedhere
    \end{align*}
\end{proof}

We are finally in a position to show that the sets $\CExpF(\vep)$ are invariant under the action of the pushforward map induced in $\SM(X)$ by $\phi_t$ for every $t \in \R$ and $\vep \in (0,1)$. 

\begin{lemma}\label{lemma_invariance_smaller_radius}
    Let $\phi$ be a regular flow on $X$ and $\vep \in (0,1)$. For every $\de \in (0,\vep)$, and $t \in \R \setminus \{0\}$ it holds that
    $$\phi_t*\ExpF(\vep,\vep) \subseteq \ExpF(\de,\de).$$
    In particular, $\CExpF(\vep)$ is $\phi_t*$-invariant for every $\vep \in (0,1)$ and $t \in \R$.
\end{lemma}

\begin{proof}
    Let $\vep>0$. Fix $x \in X$, $\de \in (0,\vep)$, and $t \in \R \setminus \{0\}$. By Lemma \ref{lemma_rep_rational}, for every $y \in \Ga_{\de}^\de(x)$ there exists $g \in Rep(\vep)$ such that $\phi_{g^{-1}(-t)} \in \Q$ and for every $s \in \R$ we have $d(\phi_s(x), \phi_{g(s)}(z)) \leq \vep$.

    It follows from Lemma \ref{lemma_translation_bowen_ball} that $\phi_{-t}(y) \in \Ga_\vep^\vep(\phi_{g^{-1}(-t)}(x))$. As a consequence, we obtain that
    $$\phi_{-t}(\Ga_{\de}^\de(x)) \subseteq \bigcup_{r \in \Q} \Ga_\vep^\vep(\phi_r(x)).$$

    Hence, if $\mu \in \ExpF(\vep,\vep)$, then we have 
    \[\phi_t*\mu (\Ga_{\de}^\de(x)) = \mu (\phi_{-t}(\Ga_{\de}^\de(x)) \leq \sum_{r \in \Q} \mu (\Ga_\vep^\vep(\phi_r(x))) = 0.\]

     It remains to prove that $\CExpF(\vep)$ is $\phi_t*$-invariant for every $t \in \R$. Fix $\mu \in \CExpF(\vep)$ and $t \in \R$. Let $\de \in (0,\vep)$ and $\eta\in (\de,\vep)$. Since $\mu \in \ExpF(\eta,\eta)$, it follows that $\phi_t*\mu \in \ExpF(\de,\de)$, and therefore,
    \[\phi_t*\CExpF(\vep) \subseteq \CExpF(\vep). \qedhere\]
\end{proof}

\begin{lemma}
    If $\phi$ is a regular flow on $X$, then $\CExpF(\vep) \subseteq \SM(X)$ is convex for every $\vep \in (0,1)$.
\end{lemma}

\begin{proof}
    Let $\vep \in (0,1)$ and $\mu,\nu \in \CExpF(\vep)$. For every $\de \in (0,\vep)$, $s \in [0,1]$, and $x \in X$ we have
    \[\left(s\mu + \left(1-s\right)\nu\right)\left(\Ga_\de^\de\left(x\right)\right) = s\mu(\Ga_\de^\de(x)) + (1-s)\nu(\Ga_\de^\de(x)) = 0.\]
    Hence, $s\mu + \left(1-s\right)\nu \in \CExpF(\vep)$.
\end{proof}

\begin{definition}
    Given $\vep \in (0,1)$, the \textbf{$\vep$-measure expansive center} of a regular flow $\phi$ on a compact metric space $X$ is 
    $$E(\phi,\vep) = \bigcup_{\nu \in \CExpF(\vep)} \supp(\nu).$$
\end{definition}

The next lemma can be proved using the same argument of \cite[Lemma 4.3]{LMS}. 
\begin{lemma}\label{lemma_residual}
    Let $\phi$ be a regular flow on $X$ and $\vep \in (0,1)$. If $\CExpF(\vep) \neq \emptyset$, then there exists a residual set $\SR \subseteq \CExpF(\vep)$ such that $\supp(\mu) = E(\phi,\vep)$ for every $\mu \in \SR$. In particular, $E(\phi,\vep)$ is compact.
\end{lemma}

Our last ingredient to prove Theorem \ref{thm_supp_invariant_g_de_si} is the proof that $E(\phi,\vep)$ is $\phi$-invariant. This is a consequence of the following lemma.

\begin{lemma}\label{lemma_measure_center_invariant}
    If $\phi$ is a regular flow on $X$, then $E(\phi,\vep)$ is $\phi$-invariant for every $\vep \in (0,1)$.
\end{lemma}

\begin{proof}
    Recall that for every homeomorphism $f \colon X \to X$ and $\mu \in \SM(X)$ it holds that $f(\supp(\mu)) = \supp(f*\mu)$. Fix $\vep \in (0,1)$.

    If $\CExpF(\vep) =  \emptyset$, then it is trivially $\phi$-invariant. Now, assume that $\CExpF(\vep) \neq \emptyset$. By Lemma \ref{lemma_residual}, there exists a residual set $\SR \subseteq \CExpF(\vep)$ such that $\supp(\mu) = E(\phi,\vep)$ for every $\mu \in \SR$. By Lemma \ref{lemma_invariance_smaller_radius}, we have $\phi_t*\CExpF(\vep) \subseteq \CExpF(\vep)$ for every $t \in \R$. Therefore, if we take $\mu \in \SR$, then for every $t \in \R$ it holds
    \begin{equation}\label{eq_phit_contained}
        \phi_t(E(\phi,\vep)) = \phi_t(\supp(\mu)) = \supp(\phi_t*\mu) \subseteq E(\phi,\vep).
    \end{equation}
    On the other hand, since the time-$t$ maps are homeomorphisms, it follows from \eqref{eq_phit_contained} that for every $t \in \R$ it holds
    \begin{equation}\label{eq_phit_contains}
        E(\phi,\vep) = \phi_{-t} \circ \phi_t(E(\phi,\vep)) \subseteq \phi_{-t}(E(\phi,\vep)).
    \end{equation}

    Since \eqref{eq_phit_contained} and \eqref{eq_phit_contains} hold for every $t \in \R$, we conclude that $\phi_t(E(\phi,\vep)) = E(\phi,\vep)$ for every $t \in \R$, which completes the proof.
\end{proof}

Finally, let us provide the reader with a proof for Theorem \ref{thm_supp_invariant_g_de_si}.

\begin{proof}[Proof of Theorem \ref{thm_supp_invariant_g_de_si}]
    Let $\SE \subset \SM(X)$ denote the set of all expansive measures for $\phi$. Note that
    $$\SE = \bigcup_{n \in \N} \CExpF\left(\frac{1}{n}\right).$$

    Recall that by Remark \ref{remark_gdelta} $\CExpF\left(\frac{1}{n}\right)$ is a $G_\de$-subset of $\SM(X)$ for every $n \in \N$, and consequently, $\SE$ is a $G_{\de \si}$-subset of $\SM(X)$.
    
    By Theorem \ref{thm_al_vep_ball}, if $\mu \in \SM(X)$ is an expansive measure for $\phi$, then there exists $\vep \in (0,1)$ such that $\mu \in \CExpF(\vep)$. By Lemma \ref{lemma_residual}, there exists a residual (in particular, dense) set $\SR \subseteq \CExpF(\vep)$ such that $\supp(\nu) = E(\phi,\vep)$ for every $\nu \in \SR$. We complete the proof by observing that, by Lemma \ref{lemma_measure_center_invariant}, $E(\phi,\vep)$ is a $\phi$-invariant set.
\end{proof}

\section{Measure of Stable Classes}\label{section_thmD}

In this section we shall prove Theorem \ref{thm:classes}. Before proceeding with the proof, let us recall the concept of stable class. For any $x\in X$, the \textbf{strong stable set of $x$} is defined as $$ W^{ss}(x)=\{y\in X: \exists h\in Rep \ \text{s.t.} \ \lim\limits_{t\to\infty}d(\phi_{t}(x),\phi_{h(t)}(y))=0\},$$and the \textbf{stable set of $x$} is defined as $$W^s(x)=\bigcup_{t\in\R}W^{ss}(\phi_t(x)).$$

\begin{definition}Let $\phi$ be a continuous flow on $X$. We
 say that a set $A\subseteq X$ is a \textbf{stable class of $\phi$} if $A = W^s(x)$ for some $x\in X$. 
\end{definition}

\begin{proof}[Proof of Theorem \ref{thm:classes}]
Let $A$ be a stable class of $\phi$ and let $x_0\in X$ be such that $A=W^s(x_0)$. Fix $\mu \in \SM_\phi(X)$ an ergodic measure with positive entropy. By Theorem \ref{entropythm}, $\mu$ is positively expansive for $\phi$, and consequently, by Theorem \ref{thmA_forward}
 there exists $\eps>0$ such that $\mu(\Gamma^+_{\eps}(x))=0$ for every $x\in X$.

Let $0<\delta< \frac{\eps}{2}$. By Lemma \ref{lem: Bowen}, there is $\te>0$ such that if $y\in \phi_{[-\te,\te]}(x)$ for some $x \in X$, then $d(\phi_t(x),\phi_t(y))\leq\delta$ for every $t\in \R$, and consequently, $y\in \Gamma^+_{\delta}(x)$. 

Now, consider the following family of orbit arcs $$\SU=\{\phi_{[-\te,\te]}(y): y\in O(x_0)\}.$$ Since $\SU$ is a cover of $O(x_0)$ by orbit arcs, there are countably many points $y_i \in O(x_0)$ such that $$O(x_0)\subset \bigcup_{i\in \N} \phi_{[-\te,\te]}(y_i).$$

Fix $z\in W^s(x_0)$. By definition, there is $y\in O(x_0)$ such that $z\in W^{ss}(y)$. Hence, there are $h\in Rep$ and $s_0>0$, such that for every $s \geq s_0$ it holds that $$d(\phi_{h(s)}(z),\phi_s(y))\leq \delta.$$ We observe that $\phi_{s_0}(y) \in \phi_{[-\te,\te]}(y_i)$ for some $i \in \N$. As a consequence, we obtain that for every $t \in \R^+$ it holds that
$$d(\phi_{h(s_0+t)}(z),\phi_t(y_i))\leq d(\phi_{h(s_0+t)}(z),\phi_{s_0+t}(y))+ d(\phi_{s_0+t}(y),\phi_{t}(y_i))\leq \de+\de\leq\eps.$$
Hence, by writing $g(t)=h(s_0+t)-h(s_0)$, one can easily see that $\phi_{h(s_0)}(z)\in \Gamma^+_{\eps}(y_i)$. 
The implies that
$$W^s(x_0)\subset \bigcup_{i\in \N}\bigcup_{t\geq 0}\phi_{-t}(\Gamma^+_{\eps}(y_i)).$$

Since $\mu$ is positively expansive and by the invariance of $\mu$ one has $$\mu(\phi_{-t}(\Gamma^+_{\eps}(y_i)))=0,$$ for every $t\in \R^+$ and $i\in \N$. Hence, by arguing as in Lemma \ref{lemma_invariance_smaller_radius}, we conclude $$\mu(A)\leq \sum_{i\in \N}\sum_{t\in \Q\cap[0,+\infty]} \mu(\phi_{-t}(\Gamma^+_{\eps}(y_i)))=0.$$ 

Lastly, suppose that $\phi$ has only countable classes $\{A_n\}_{n\in \N}$. If there is $\mu \in \SM_\phi(X)$ is an ergodic measure with positive entropy, then $\mu(A_n)=0$ for every $n \in \N$. Since the collection of stable classes forms a partition of $X$, we would have $$\mu(X)=\sum_{n\in \N}\mu(A_n)=0,$$
which is a contradiction. Therefore, if $\phi$ has a measure $\mu \in \SM_\phi(X)$ with positive entropy, then it must have uncountable many stable classes. 
\end{proof}

\section*{Acknowledgments.}
We would like to thank Prof. Yun Yang for kindly pointing out an issue in one of the lemmas in a previous version of this work. 
Eduardo Pedrosa was financed in part by the
 Coordenação de Aperfeiçoamento de Pessoal de Nível Superior - (CAPES) Brasil - Finance Code 001, and by the Conselho
 Nacional de Desenvolvimento Científico e Tecnológico (CNPq) Brasil.
Alexandre Trilles was funded by NCN Sonata Bis grant no. 2019/34/E/ST1/00082.

\begin{table}[h]
\begin{tabularx}{\linewidth}{p{1.5cm}  X}
\includegraphics [width=1.8cm]{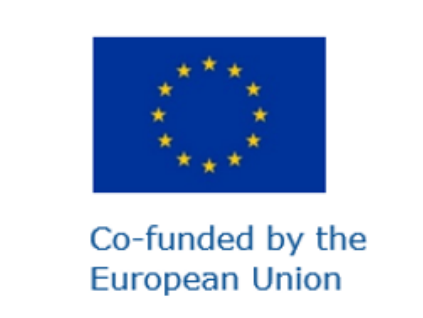} &
\vspace{-1.5cm}
This research is part of a project that has received funding from
the European Union's European Research Council Marie Sklodowska-Curie Project No. 101151716 -- TMSHADS -- HORIZON--MSCA--2023--PF--01.\\
\end{tabularx}
\end{table}

\bibliographystyle{plain}
\bibliography{references}

@article{AM,
   author={Arbieto, A. and Morales, C. A.},
   title={Some properties of positive entropy maps},
   journal={Ergodic Theory Dynam. Systems},
   volume={34},
   year={2014},
   number={3},
   pages={765--776},
   issn={0143-3857},
   review={\MR{3199792}},
   doi={10.1017/etds.2012.162},
}

@article{ACCV,
   author={Artigue, A. and Carvalho, B.
    and Cordeiro, W.
    and Vieitez, J.},
   title={Countably and entropy expansive homeomorphisms with the shadowing
   property},
   journal={Proc. Amer. Math. Soc.},
   volume={150},
   year={2022},
   number={8},
   pages={3369--3378},
   issn={0002-9939},
   review={\MR{4439460}},
   doi={10.1090/proc/15326},
}

@article{BW,
  author={Bowen, R. and Walters, P.},
  title={Expansive one-parameter flows},
  journal={J. Differential Equations},
  volume={12},
  year={1972},
  pages={180--193},
}

@incollection {BK,
    AUTHOR = {Brin, M. and Katok, A.},
     TITLE = {On local entropy},
 BOOKTITLE = {Geometric dynamics ({R}io de {J}aneiro, 1981)},
    SERIES = {Lecture Notes in Math.},
    VOLUME = {1007},
     PAGES = {30--38},
 PUBLISHER = {Springer, Berlin},
      YEAR = {1983},
      ISBN = {3-540-12336-9},
   MRCLASS = {58F11 (28D20)},
  MRNUMBER = {730261},
MRREVIEWER = {R.\ Dil\~ao},
       DOI = {10.1007/BFb0061408},
       URL = {https://doi.org/10.1007/BFb0061408},
}

@article{CadreJacob2003,
  author={Cadre, B. and Jacob, P.},
  title={On pairwise sensitive homeomorphisms},
  journal={Fund. Math.},
  volume={181},
  year={2004},
  pages={213--228},
}

@article{CarrascoMorales2014,
  author={Carrasco-Olivera, D. and Morales, C. A.},
  title={Expansive measures for flows},
  journal={J. Differential Equations},
  volume={256},
  year={2014},
  pages={2246--2260},
}

@article{CC,
   author={Carvalho, B. and Cordeiro, W.},
   title={$n$-expansive homeomorphisms with the shadowing property},
   journal={J. Differential Equations},
   volume={261},
   year={2016},
   number={6},
   pages={3734--3755},
   issn={0022-0396},
   review={\MR{3527644}},
   doi={10.1016/j.jde.2016.06.003},
}

@article{Denker2006,
  author={Denker, M.},
  title={Measure-theoretic aspects of sensitive dynamical systems},
  journal={Nonlinearity},
  volume={19},
  year={2006},
  number={7},
  pages={1695--1711},
}

@article {Katok80,
    AUTHOR = {Katok, A.},
     TITLE = {Lyapunov exponents, entropy and periodic orbits for
              diffeomorphisms},
   JOURNAL = {Inst. Hautes \'Etudes Sci. Publ. Math.},
  FJOURNAL = {Institut des Hautes \'Etudes Scientifiques. Publications
              Math\'ematiques},
    NUMBER = {51},
      YEAR = {1980},
     PAGES = {137--173},
      ISSN = {0073-8301,1618-1913},
   MRCLASS = {28D20 (58F11 58F15)},
  MRNUMBER = {573822},
MRREVIEWER = {R.\ L.\ Adler},
       URL = {http://www.numdam.org/item?id=PMIHES_1980__51__137_0},
}

@book{KatokHasselblatt,
  author={Katok, A. and Hasselblatt, B.},
  title={Introduction to the Modern Theory of Dynamical Systems},
  series={Encyclopedia of Mathematics and its Applications},
  volume={54},
  publisher={Cambridge University Press},
  year={1995},
}

@article{Ko,
   author={Komuro, M.},
   title={One-parameter flows with the pseudo-orbit tracing property},
   journal={Monatsh. Math.},
   volume={98},
   year={1984},
   number={3},
   pages={219--253},
   issn={0026-9255},
   review={\MR{0774757}},
   doi={10.1007/BF01507750},
}

@article{LMS,
  author={Lee, K. and Morales, C. A. and Shin, B.},
  title={On the set of expansive measures},
  journal={Commun. Contemp. Math.},
  volume={19},
  year={2017},
  number={6},
  pages={1750086},
}

@book {MS,
    AUTHOR = {Morales, Carlos A. and Sirvent, V\'ictor F.},
     TITLE = {Expansive measures},
    SERIES = {Publica\c c\~oes Matem\'aticas do IMPA. [IMPA Mathematical
              Publications]},
      NOTE = {29$\sp {\rm o}$ Col\'oquio Brasileiro de Matem\'atica. [29th
              Brazilian Mathematics Colloquium]},
 PUBLISHER = {Instituto Nacional de Matem\'atica Pura e Aplicada (IMPA), Rio
              de Janeiro},
      YEAR = {2013},
     PAGES = {viii+89},
      ISBN = {978-85-244-0360-6},
   MRCLASS = {37-01 (37B05)},
 }

@article{OR,
  author={Oprocha, P. and Rego, E.},
  title={Local Shadowing and Entropy for Homeomorphisms},
  journal={pre-print},
  volume={arXiv:2502.10752},
}

@article{Ruggiero1996,
  author={Ruggiero, R.},
  title={Expansive flows},
  journal={Manuscripta Math.},
  volume={89},
  year={1996},
  pages={281--293},
}

@article{SV,
   author={Sun, W. and Vargas, E.},
   title={Entropy of flows, revisited},
   journal={ Bol. Soc. Bras. Mat.},
   volume={30},
   year={1999},
   pages={315–333},
   doi={10.1007/BF01239009},
}

@article{Utz1950,
  author={Utz, W. R.},
  title={Unstable homeomorphisms},
  journal={Proc. Amer. Math. Soc.},
  volume={1},
  year={1950},
  pages={769--774},
}

@article {JCWZ,
    AUTHOR = {Ji, Yong and Chen, Ercai and Wang, Yunping and Zhao, Cao},
     TITLE = {Bowen entropy for fixed-point free flows},
   JOURNAL = {Discrete Contin. Dyn. Syst.},
  FJOURNAL = {Discrete and Continuous Dynamical Systems. Series A},
    VOLUME = {39},
      YEAR = {2019},
    NUMBER = {11},
     PAGES = {6231--6239},
     
}

@article{F_expansive,
title = {F-expansivity for Borel measures},
journal = {Journal of Differential Equations},
volume = {261},
number = {10},
pages = {5350-5370},
year = {2016},
issn = {0022-0396},
doi = {https://doi.org/10.1016/j.jde.2016.08.004},
url = {https://www.sciencedirect.com/science/article/pii/S002203961630208X},
author = {H. {Villavicencio Fernández}}
}

\end{document}